\documentclass[preprint]{elsarticle}
\usepackage{caption}
\usepackage{epsfig}
\usepackage{amssymb}
\usepackage{amsthm}
\newtheorem{theorem}{Theorem}
\newtheorem{lemma}{Lemma}
\newtheorem{remark}{Remark}
\newtheorem{corollary}{Corollary}
\newtheorem{proposition}{Proposition}
\def\Frac#1#2{\frac{\displaystyle{#1}}{\displaystyle{#2}}}

\begin{document}

\begin{frontmatter}

\title{Monotonicity properties and bounds for the chi-square and gamma distributions.}

\author{Javier Segura}
\address{Departamento de Matem\'aticas, Estad\'{\i}stica y Computaci\'on.

Facultad de Ciencias. 

Universidad de Cantabria. 

39005-Santander, SPAIN. 

javier.segura@unican.es
}




\begin{abstract}
The generalized Marcum functions $Q_{\mu}(x,y)$ and $P_{\mu}(x,y)$ 
have as particular cases the non-central $\chi^2$ 
and gamma cumulative distributions, which become central distributions (incomplete
gamma function ratios) when the 
non-centrality parameter $x$ is set to zero. We analyze monotonicity and convexity 
properties for the generalized Marcum functions and for ratios of Marcum functions
of consecutive parameters (differing in one unity) and we obtain upper and lower
bounds for the Marcum functions. These bounds are proven to be sharper than previous 
estimations for a wide range of the parameters. 
Additionally we show how to build convergent sequences of upper and lower bounds. The
particularization to incomplete gamma functions, together with some additional bounds
obtained for this particular case, lead to combined bounds which improve previously 
exiting inequalities. 
\end{abstract}


\end{frontmatter}


\section{Introduction and definitions}

Generalized Marcum functions are defined as
\begin{equation}
\label{eq:defQmu}
Q_{\mu} (x,y)=\displaystyle x^{\frac12 (1-\mu)} \int_y^{+\infty} t^{\frac12 (\mu -1)} e^{-t-x} I_{\mu -1} \left(2\sqrt{xt}\right) \,dt,
\end{equation}
where $\mu >0$ and $I_\mu(z)$ is the modified Bessel function 
\cite[10.25.2]{Olver:2010:BF}. The generalized Marcum $Q$-function is an important function
 used in radar detection and communications. They also occur in statistics and probability theory, where they are called
non-central chi-square or non central gamma cumulative distributions (see \cite{Gil:2014:COT} and references
cited therein).

The complementary function is,
\begin{equation}
\label{eq:defPmu}
P_{\mu} (x,y)=\displaystyle x^{\frac12 (1-\mu)}\int_0^{y} t^{\frac12 (\mu -1)} e^{-t-x} I_{\mu -1} \left(2\sqrt{xt}\right) \,dt,
\end{equation}
and for $\mu>0$ and $x,y\ge 0$ we have
\begin{equation}\label{eq:PQcompl}
P_{\mu}(x,y)+Q_{\mu} (x,y)=1.
\end{equation}

The central chi-square or gamma cumulative distributions $P(a,y)$ and $Q(a,y)$ are a particular case of the non-central
distributions with non-centrality parameter $x$ equal to zero: $P(a,y)=P_{a}(0,y)$, $Q(a,y)=P_{a}(0,y)$. 
 These are functions related to the incomplete gamma function ratios
\begin{equation}\label{eq:int02}
P(a,y)=\frac{1}{\Gamma(a)}\gamma(a,y), \quad
Q(a,y)=\frac{1}{\Gamma(a)}\Gamma(a,y),
\end{equation}
where
\begin{equation}\label{eq:int01}
\gamma(a,y)=\int_0^y t^{a-1} e^{-t}\,dt, \quad
\Gamma(a,y)=\int_y^{\infty} t^{a-1} e^{-t}\,dt .
\end{equation}

There are other notations for the generalized Marcum function in the literature. Among them, probably the most
 popular is the following
\begin{equation}
\widetilde{Q}_{\mu} (\alpha,\beta)=\alpha^{1-\mu}\int_\beta^{+\infty}t^{\mu} e^{-(t^2+\alpha^2)/2} I_{\mu -1}(\alpha t) dt,
\end{equation}
where we have added a tilde in the definition to distinguish it from the definition we are using (\ref{eq:defQmu}). For
$\mu=1$ this coincides with the original definition of the Marcum $Q$-function \cite{Marcum:1960:AST}.
The relation with the notation we use is simple:
\begin{equation}
Q_{\mu}(x,y)=\widetilde{Q}_{\mu} (\sqrt{2x},\sqrt{2y}),
\end{equation}
and similarly for the $P$ function.

Our notation for the $P$ and $Q$ functions is directly related to the $\chi^2$ cumulative distribution
function ${\mathbf P}(x;k,\lambda)$ by the relation
\begin{equation}
{\mathbf P}(x;k,\lambda)=P_{k/2}(\lambda /2, x/2).
\end{equation}
In the $\chi^2$ cumulative distribution, integer values of $k$ appear. We consider the more general case of real
positive $k$,
 in which case the distributions $Q$ and $P$ are also called noncentral gamma distributions.

In this paper we study monotonicity and convexity properties for the cumulative distribution functions 
defined by 
(\ref{eq:defPmu}) and (\ref{eq:PQcompl}), and for ratios of functions of consecutive orders, as well as 
bounds these functions.

We start by summarizing some basic properties of $P_{\mu}(x,y)$ and $Q_{\mu}(x,y)$, including monotonicity and convexity
(section \ref{basic}).
Then, in section \ref{noncentral}, monotonicity properties and bounds for the ratios of functions of consecutive orders $P_{\mu+1}(x,y)/P_{\mu}(x,y)$
and $Q_{\mu+1}(x,y)/Q_{\mu}(x,y)$ are obtained, which lead to bounds for $P_{\mu}(x,y)$ and $Q_{\mu}(x,y)$ 
in terms of two modified Bessel functions of consecutive orders. 
It is also discussed how to obtain convergent sequences of 
upper and lower bounds. Particularizing for $x=0$ we obtain
bounds for the central case; additional bounds for the central case are also obtained from new monotonicity properties
for the incomplete gamma function ratios (section \ref{central}).  Finally, in section \ref{compar} we
compare our new bounds with previous bounds and we conclude that the new bounds for the noncentral distributions 
improve previous results (of similar complexity) 
for a wide range of the parameters, and that combined bounds in the central case improve
previous existing bounds in the full range of parameters.

\section{Basic properties}
\label{basic}

Considering integration by parts together with the relation $z^\mu I_{\mu-1}(z)=\frac{d}{dz}\left(z^\mu I_\mu(z)\right)$
we get

\begin{equation}
\label{RecQ}
\begin{array}{ll}
\displaystyle{Q_{\mu+1}(x,y)=Q_{\mu} (x,y)+\left(\Frac{y}{x}\right)^{\mu/2} e^{-x-y} I_{\mu} (2\sqrt{xy})},\\[8pt]
\displaystyle{P_{\mu+1}(x,y)=P_{\mu} (x,y)-\left(\Frac{y}{x}\right)^{\mu/2} e^{-x-y} I_{\mu} (2\sqrt{xy})},
\end{array}\end{equation}
where the relation for $Q$ holds for all real $\mu$ while for $P$ it only holds for $\mu>0$.

From this, we can obtain the following recurrence
\begin{equation}
\label{TTRR}
y_{\mu+1} -(1+c_\mu) y_\mu + c_\mu y_{\mu-1}=0,\quad c_\mu=\sqrt{\frac{y}{x}}\Frac{I_\mu \left(2\sqrt{xy}\right)}{I_{\mu-1}\left(2\sqrt{xy}\right)},
\end{equation}
satisfied both by $Q_\mu (x ,y)$ and $P_\mu (x ,y)$.

From \cite[10.41.1]{Olver:2010:BF} we see that $c_{\mu}={\cal O}(\mu^{-1})$ as $\mu\rightarrow +\infty$; Perron-Kreuser theorem 
\cite[Theorem 4.5]{Gil:2007:NSF}
guarantees that the recurrence has a minimal solution such that $y_{\mu+1}/y_{\mu}\sim c_{\mu}$ as $\mu\rightarrow +\infty$.
This corresponds with the $P_{\mu}(x,y)$ function. The dominant
solutions of the recurrence are such that $y_{\mu+1}/y_{\mu}\sim 1$, and this corresponds to the case of the  $Q_{\mu}(x,y)$
function. Therefore

\begin{equation}
\label{limimu}
\lim_{\mu\rightarrow +\infty}\Frac{1}{c_{\mu}(x,y)}\Frac{P_{\mu}(x,y)}{P_{\mu-1}(x,y)}=1,\,
\lim_{\mu\rightarrow +\infty}\Frac{Q_{\mu}(x,y)}{Q_{\mu -1}(x,y)}=1.
\end{equation}
Later we will prove that $P_{\mu}/P_{\mu-1}(x,y)<c_{\mu}(x,y)\le\mu/y$ and $P_{\mu}/P_{\mu-1}(x,y)<1$.

Applying $n$ times the backward recurrence for the $P$-function (\ref{RecQ}) we have
\begin{equation}
\label{sumaPs}
P_{\mu}(x,y)=P_{\mu+n+1}(x,y)+\displaystyle\sum_{k=0}^{n}F_{\mu+n}(x,y),\,\mu>0,
\end{equation}
where
\begin{equation}
\label{Fmu}
F_{\mu}(x,y)=\left(\Frac{y}{x}\right)^{\mu /2}e^{-x-y}I_{\mu}(2\sqrt{xy}). 
\end{equation}
And because $P_{\mu}(x,y)\rightarrow 0$ as $\mu\rightarrow +\infty$,
\begin{equation}
\label{seriey}
P_{\mu}(x,y)=\displaystyle\sum_{k=0}^{\infty}F_{\mu+k}(x,y)=e^{-x-y}
\displaystyle\sum_{k=0}^{\infty}\left(\Frac{y}{x}\right)^{\frac{\mu +k}{2}} I_{\mu+k}(2\sqrt{xy}).
\end{equation}

Taking the derivative with respect to $y$ in (\ref{eq:defQmu}) and using (\ref{RecQ}) we have
\begin{equation}
\label{de1}
\Frac{\partial Q_\mu(x,y)}{\partial y}=Q_{\mu-1}(x,y)-Q_{\mu}(x,y) .
\end{equation}

Taking the derivative with respect to $x$, and using the relation $I_{\nu}'(z)=I_{\nu+1}(z)+\frac{\nu}{z}I_{\nu}(z)$, we obtain
\begin{equation}
\label{de2}
\Frac{\partial Q_\mu(x,y)}{\partial x}=Q_{\mu+1}(x,y)-Q_{\mu}(x,y)=-\Frac{\partial Q_{\mu +1}(x,y)}{\partial y} .
\end{equation}

The same relations hold for the $P$-functions.

Using (\ref{RecQ}) we see that $Q_\mu (x,y)$ ($P_\mu (x,y)$) is an increasing (decreasing) 
function of $x$ and a decreasing (increasing) function of $y$. 
With respect to $\mu$, $Q_\mu(x,y)$ is increasing and $P_\mu(x,y)$ is decreasing \cite{Sun:2008:IFT}.

We have the following particular values
\begin{equation}\label{qmval}
\begin{array}{ll}
Q_{\mu} (x, 0)=1,&Q_{\mu}(x ,+\infty)=0,\\[8pt]
Q_{\mu} (0,y)=Q_\mu(y),\quad& Q_{\mu}(+\infty,y)=1,\\[8pt]
Q_{+\infty} (x ,y )=1,&
\end{array}
\end{equation}
and similar complementary relations for $P_\mu(x,y)$.

Usually, we will consider Marcum functions $Q_{\mu}(x,y)$ and $P_{\mu}(x,y)$ for positive $\mu$, where
both functions are positive and (\ref{eq:PQcompl}) holds. However, in order to extend the validity for some 
of the bounds we will obtain for the $Q$-function, it will be
 convenient to consider smaller values of $\mu$.
The $Q$-function defined in Eq. (\ref{eq:defQmu})
is continuous for any real value of $\mu$; the same is not true for $P_{\mu}(x,y)$ which is only defined for 
$\mu\ge 0$ and is discontinuous from the right at $\mu=0$. It 
is easy to show that, using the definition of the $Q$ and $P$ functions and \cite[4.16 (16)]{Erdelyi:1954:TIT}
$$
Q_0 (x,y)+ P_0 (x,y) =1-e^{-x},
$$
and (\ref{eq:PQcompl}) does not hold because of the discontinuity of $P_{\mu}(x,y)$ at $\mu=0$. Also, the recurrence 
(\ref{RecQ})  does not hold for $\mu=0$; instead, we have $P_1(x,y)=P_0 (x,y)-F_0 (x,y) +e^{-x}$.

For negative $\mu$, $Q_{\mu}(x,y)$ satisfies the same recurrence relations as for positive $\mu$ and the series
in terms of incomplete gamma functions (obtained by using the Frobenius series for $I_{\mu-1}(z)$ in the
definition (\ref{eq:defQmu})) also holds, namely:
\begin{equation}
\label{qserq}	
Q_{\mu}(x,y)=e^{-x}\displaystyle\sum_{k=0}^{\infty}\Frac{x^k}{k!}Q_{\mu+k}(y).
\end{equation}

We observe that \cite[10.27.2]{Olver:2010:BF} 
$$
I_{-\nu}(z)=I_{\nu}(z)+\Frac{2}{\pi}\sin\nu\pi K_{\nu}(z)
$$
implies that $Q_{\mu}(x,y)>0$ if $\mu\in [-2k,-2k+1]$, $k\in {\mathbb N}$, but that $Q_{\mu}(x,y)$ can be negative for other
negative values of $\mu$, and particularly when $x$ is small enough; indeed $Q_{\mu}(0,y)=Q_{\mu}(y)<0$ for these values 
of $\mu$ because $Q_{\mu}(y)=\Gamma (\mu,y)/\Gamma (\mu)$ and $\Gamma (\mu)<0$. However, for sufficiently
large $x$, $Q_{\mu}(x,y)$ becomes positive. In particular, we have the following:

\begin{proposition}
Let $\mu_0 \in (-1,0)$ and $y>0$, if
\begin{equation}
\label{Qmuco}
x\ge y^{\mu_0}e^{-y}/\Gamma (\mu_0+1,y)-1
\end{equation}
then $Q_{\mu_0}(x,y)>0$.

\end{proposition}
\begin{proof}
Using Eq. (\ref{qserq}), we see that $Q_{\mu_0}(x,y)>0$ if $Q_{\mu_0}(y)+xQ_{\mu_0+1}(y)>0$ 
(because $Q_{\mu_0+k}(y)>0$, $k=2,3,\ldots$). Then,
considering the recurrence (\ref{RecQ}) for $x=0$ ($F_{\mu}(0,y)=y^{\mu}e^{-y}/\Gamma (\mu+1)$) the result
is proved.  
\end{proof}

Using an inequality we will prove later (Eq. (\ref{boundGAM})), we have:
\begin{corollary}
\label{primomo}
Let $y>0$ and $\mu_0 \in (-1,0)$,  if
$$
x\ge L_{\mu_0}(y)=\Frac{\sqrt{(y-\mu_0-2)^2+4y}-y-\mu_0}{2y}
$$
then $Q_{\mu}(x,y)>0$ for all $\mu\ge \mu_0$.
\end{corollary}
\begin{proof}
Considering the bound (\ref{boundGAM}) it is immediate to check that 
the condition $x\ge L_{\mu_0}(y)$ implies the condition (\ref{Qmuco}) and therefore
$Q_{\mu_0}(x,y)>0$. But because $L_{\mu}(y)$ is decreasing as a function of $\mu$, then it is
also true that $x>L_{\mu}(y)$ holds for any $\mu\ge \mu_0$ and therefore
$Q_{\mu}(x,y)>0$ for all $\mu\ge \mu_0$.
\end{proof}
\begin{corollary}
\label{xyuno}
If $\mu\ge -2$ and $xy\ge 1$ then $Q_{\mu}(x,y)>0$.
\end{corollary}
\begin{proof}
For $\mu\in [-2,-1]$ we have $Q_{\mu}(x,y)>0$ for any positive $x$ and $y$, and if $x\ge F_{-1}(y)=1/y$
then we have $Q_\mu (x,y)>0$ for all $\mu>-1$ (Corollary  \ref{primomo}).
\end{proof}

\subsection{Convexity properties}
\label{convex}

Next we study the convexity properties of Marcum functions and, more specifically, we bound the values of $x$ and $y$ for 
the inflection points. This is important information for the construction of algorithms for the inversion
of Marcum functions \cite{Gil:2014:THA}.

Using (\ref{de1}), (\ref{de2}) and (\ref{RecQ}), we obtain
\begin{equation}
\begin{array}{l}
\Frac{\partial^2 Q_{\mu}(x,y)}{\partial x^2}=(c_{\mu+1}(x,y)-1)F_{\mu}(x,y),\\
\\
\Frac{\partial^2 Q_{\mu}(x,y)}{\partial y^2}=(c_{\mu-1}(x,y)-1)F_{\mu-2}(x,y),
\end{array}
\end{equation}
with $c_\mu(x,y)$ as defined in (\ref{TTRR}) and $F_{\mu}$ as in (\ref{Fmu}).
Because $I_{\mu}(t)>0$ for $t>0$ and $\mu\ge -1$ we have that $F_{\mu}(x,y)>0$ if $x\ge 0$, $y>0$, and $\mu\ge -1$ (the
case $x\rightarrow 0^{+}$ is easy to check using the limiting form \cite[10.30.1]{Olver:2010:BF}); then,
the sign of the derivative $\partial^2 Q_{\mu}(x,y)/\partial x^2$ is the same as the sign of $c_{\mu+1}(x,y)-1$ if
$\mu\ge -1$ (and $x\ge 0$, $y>0$) while the sign of $\partial^2 Q_{\mu}(x,y)/\partial x^2$ 
coincides with the sign of $c_{\mu-1}(x,y)-1$ if $\mu\ge 1$.

For analyzing the convexity properties, and also for analyzing the bounds we will later obtain, Lemma \ref{coefi}
will be useful. For proving that result, we will need to consider the following bounds:

\begin{lemma}
\label{lemo}
\begin{equation}
\label{eco}
\Frac{t}{f_{\nu}(t)}<\Frac{I_{\nu}(t)}{I_{\nu -1}(t)}<\Frac{t}{f_{\nu-1}(t)},\nu\ge 0,
\end{equation}
where
\begin{equation}
f_{\lambda}(t)=(\lambda+\sqrt{\lambda^2+x^2})^{-1}.
\end{equation}
For $\nu\ge 1/2$ a sharper upper bound is
\begin{equation}
\label{mejora}
\Frac{I_{\nu}(t)}{I_{\nu -1}(t)}<\Frac{t}{f_{\nu-1/2}(t)}.
\end{equation}
\end{lemma}
\begin{proof}
The lower bound and the second upper bound are known results (see for instance \cite[Corollary 3]{Segura:2011:BRM} and \cite[Theorem 3]{Segura:2011:BRM} respectively). 

The first upper bound is a consequence of \cite[Theorem 4]{Segura:2011:BRM}.
Starting with the lower bound in (\ref{eco}) $b_{\nu}(x)=t/f_{\nu}(t)$ an upper bound is given by
\begin{equation}
\Frac{I_{\nu}(t)}{I_{\nu -1}(t)}<\Frac{t}{2\nu+t b_{\nu+1}(t)}=\Frac{t}{\nu-1+\sqrt{(\nu+1)^2+t^2}},\,\nu \ge 0.
\end{equation}
Obviously $\sqrt{(\nu+1)^2+t^2}\ge \sqrt{(\nu-1)^2+t^2}$ if $\nu\ge 0$, which proves the result.
\end{proof}

\begin{lemma}
\label{coefi}
The function $c_\mu (x,y)=\displaystyle\sqrt{\frac{y}{x}}\Frac{I_\mu \left(2\sqrt{xy}\right)}{I_{\mu-1}\left(2\sqrt{xy}\right)}$, $x,y>0$, 
$\mu\ge 0$,
is increasing as a function of $y$, decreasing as a function of $x$ and decreasing as a function of $\mu$. 
Specific values are:
\begin{equation}
\label{valsces}
\begin{array}{l}
c_{\mu}(0^+,y)=y/\mu,\,c_{\mu}(+\infty,y)=0,\\
c_{\mu}(x,0^+)=0,\,c_{\mu}(x,+\infty)=+\infty,\\ 
c_{+\infty}(x,y)=0.
\end{array}
\end{equation}
In addition, we have that 
\begin{equation}
\label{desc}
\begin{array}{l}
c_{\mu}(x,y)>1 \mbox { if } y>x+\mu,\,\mu\ge 0,\\
c_{\mu}(x,y)<1 \mbox { if } y<x+\mu-1/2,\,\mu\ge 1/2,\\
c_{\mu}(x,y)<1 \mbox { if } y<x+\mu-1,\,\mu\ge 0.
\end{array}
\end{equation}
\end{lemma}

\begin{proof}

The particular values (\ref{valsces}) are straightforward to verify by using well known limiting forms \cite[10.30]{Olver:2010:BF}.

Fro proving monotonicity with respect to $x$ and $y$ it will be useful to consider the Riccati equation
satisfied by $g_{\mu}(t)=I_{\mu}(t)/I_{\mu-1}(t)$. It is easy to check that (see \cite[Eq. (11)]{Segura:2011:BRM}):
\begin{equation}
\label{ricg}
g_{\mu}^{\prime}(t)=1-\Frac{2\mu -1}{t}g_{\mu}(t)-g_{\mu}(t)^2.
\end{equation}

The monotonicity with respect to $x>0$ of $c_{\mu}(x,y)$ is equivalent to the
monotonicity of $h_{\mu}(t)=g_{\mu}(t)/t$ as a function of $t>0$ for $\mu\ge 0$. From (\ref{ricg}) we have
 $th_{\mu}^{\prime}(t) =1-2\mu g_{\mu}(t)/t-g_{\mu}(t)^2$;
and because $g_{\mu}(t)>t/f_{\mu}(t)$, $t> 0$, $\mu\ge 0$ (Lemma \ref{lemo}) (see also
\cite{Laforgia:2010:SIF}) the Riccati equation for $h_{\mu}(t)$ implies that $h_{\mu}^{\prime}(t)<0$, $t>0$. 
This proves that $c_{\mu}(x,y)$ is decreasing as a function of $x> 0$.

The fact that $c_{\mu}(x,y)$ is increasing as a function of $y$ is equivalent to the fact that
$p_{\mu}(t)=tg_{\mu}(t)$ is increasing as a function of $t> 0$ for $\mu\ge 0$. From (\ref{ricg}) we
have $p_{\mu}'(t)=t\left[1-2(\mu-1)g_{\mu}(t)/t-g_{\mu}(t)^2\right]$;
and because $0<g_{\mu}(t)<t/f_{\mu-1}(t)$, $t> 0$, $\mu\ge 0$ (Lemma \ref{lemo}), using this Riccati equation
we see that $p_{\mu}'(t)>0$, $\mu\ge 0$, which implies that $c_{\mu}(x,y)$ is increasing as a function of $y>0$.

The monotonicity with respect to $\mu$ follows from the monotonicity of
 $g_{\mu}(t)=I_{\mu}(t)/I_{\mu-1}(t)$ with respect to $\mu$. To prove that this ratio is decreasing
 we take into account that $g_{\mu}(0^+)>g_{\mu'}(0^+)$
if $\mu<\mu'$ (see (\ref{valsces})) and we prove that this implies, 
on account of (\ref{ricg}), that $g_{\mu}(t)>g_{\mu'}(t)$ for all $t>0$.
We suppose the contrary and we arrive at a contradiction: let $t_c$ be the smallest positive value such that
$g(t_c)=g_{\mu}(t_c)=g_{\mu'}(t_c)$; because $g_{\mu}(t)>g_{\mu'}(t)$ if $t<t_c$ then $g'_{\mu}(t_c)<g'_{\mu'}(t_c)$;
but considering  (\ref{ricg}) we have
$$
g_{\mu'}'(t_c)=1-\Frac{2\mu' -1}{t_c}g(t_c)-g(t_c)^2<1-\Frac{2\mu -1}{t_c}g(t_c)-g(t_c)^2 =g_{\mu}'(t_c),
$$
and there is a contradiction.
\footnote{In other words, the proof of monotonicity with respect to 
$\mu$ follows from the fact that $g_{\mu}(0^+)<g_{\mu'}(0^+)$
if $\mu>\mu'$ and that, for a fixed value of $g_{\nu}(t)$, 
$$
\Frac{\partial}{\partial \nu}(g_{\nu}'(t))=-\Frac{2}{t}g_{\nu}(t),
$$
where we have derived the Riccati equation but taking $g_{\nu}(t)$ fixed (not depending on $\nu$).}

For proving the inequalities, we write the equation $c_{\mu}(x,y)=1$ as
\begin{equation}
D_{\mu}(x,y)=0,\,D_{\mu}(x,y)=\Frac{1}{z}\Frac{I_{\mu}(z)}{I_{\mu -1}(z)}-\Frac{1}{2y}=0 .
\end{equation}
Now, we consider the bounds (\ref{eco}); we can write for $\mu\ge0$:
\begin{equation}
\label{empi}
f_{\mu}(z)-\Frac{1}{2y}<D_{\mu}(x,y)<f_{\mu-1}(z)-\Frac{1}{2y},\, f_{\nu}(z)=(\nu+\sqrt{\nu^2+z^2})^{-1} .
\end{equation}
Because $f_{\nu}(z)$ is decreasing as a function of $\nu$ and continuous, there exist only one value 
$\mu^*$ such that  $\Frac{1}{2y}=f_{\mu^*}(z)$; this value is $\mu^* = y-x$. Therefore $f_{\mu}(z)-\Frac{1}{2y}>0$
if $\mu<\mu^*$ and  $f_{\mu-1}(z)-\Frac{1}{2y}<0$ if $\mu-1>\mu^*$, which implies that  
$c_{\mu}(x,y)>1$ ($D_{\mu}(x,y)>0$) if $y>x+\mu$ and
$c_{\mu}(x,y)<1$ if $y<x+\mu-1$, $\mu\ge 0$. If in (\ref{empi}) we use the upper bound (\ref{mejora}) instead, we get
that $c_{\mu}(x,y)<1$ if $y<x+\mu-1/2$ when $\mu\ge 1/2$.

\end{proof}

Lemma \ref{coefi} is key to prove the convexity properties and it will be also important in the derivation of bounds.
The monotonicity of $c_{\mu}(x,y)$ with respect to $x$ and $y$ is mentioned without proof in \cite{Saxena:1982:EOT},
where the first two inequalities of (\ref{desc}) are also proved in a different way.

Now we prove the following result

\begin{theorem}
The following convexity properties with respect to $x$ hold for $\mu\ge 0$:
\begin{equation}
\label{unol}
\Frac{\partial^2 Q_{\mu}(x,y)}{\partial x^2}\le 0 \quad \mbox{ if } \quad  0<y \le \mu +1,\,x\ge 0,
\end{equation}
where the equality only takes place for $x=0$, $y=\mu+1$.

\begin{equation}
\begin{array}{l}
\label{dosol}
\Frac{\partial^2 Q_{\mu}(x,y)}{\partial x^2}<0 \quad \mbox{ if } \quad x>y-\mu-\frac12,\\
\\
\Frac{\partial^2 Q_{\mu}(x,y)}{\partial x^2}>0 \quad \mbox{ if } \quad x<y-\mu-1.
\end{array}
\end{equation}

\end{theorem}

\begin{proof}
Because we have $c_{\mu +1}(0,y)=y/(\mu +1)$ and $c_{\mu +1}(x,y)$ decreases as a function of $x$,
 if $y<\mu +1$ we have $c_{\mu +1}(x,y)<1$ for all $x\ge 0$. This implies Eq. (\ref{unol}).

For $y>\mu+1$ there is necessarily an inflection point with respect to $x$ and only one
 (because $c_{\mu+1}$ is monotonic) and Lemma
\ref{coefi} gives information on the location of the inflexion point that leads to Eq. (\ref{dosol}).
\end{proof}

With respect to $y$ we have the following result:
\begin{theorem}
The following convexity properties with respect to $y$ hold for $\mu\ge 1$:
\begin{equation}
\label{conve_y}
\begin{array}{l}
\Frac{\partial^2 Q_{\mu}(x,y)}{\partial y^2}>0 \quad \mbox{ if } \quad y>x+\mu-1,\\
\\
\Frac{\partial^2 Q_{\mu}(x,y)}{\partial y^2}<0 \quad \mbox{ if } \quad y<x+\mu-2.
\end{array}
\end{equation}
If $\mu\ge 3/2$ the range of the last inequality can be improved
\begin{equation}
\Frac{\partial^2 Q_{\mu}(x,y)}{\partial y^2}<0 \quad \mbox{ if } \quad y<x+\mu-3/2.
\end{equation}
\end{theorem}

Observe that for $\mu\ge 1$ there is always one and only one inflection point with respect to $y$
because $c_{\mu -1}(x,y)$
is monotonically increasing as a function of $y$ and $c_{\mu-1}(x,0)=0$, $c_{\mu-1}(x,+\infty)=+\infty$
if $\mu\ge 1$. Contrarily, for $\mu<1$ this analysis can not be made, because the  
monotonicity properties of $c_{\mu-1}(x,y)$ may not hold and, additionally, the sign of the second 
derivative is not determined by the sign of $c_{\mu-1}(x,y)-1$. There are cases
for which $\partial^2 Q_{\mu}(x,y)/\partial y^2$ changes sign twice when $\mu\in (0,1)$.

\section{Non-central distributions: monotonicity and bounds}

\label{noncentral}

Next we obtain monotonicity properties for the ratios $P_{\mu +1}/P_{\mu }$ and $Q_{\mu+1}/Q_{\mu}$ and 
bounds for these ratios which, taking into account (\ref{RecQ}), will give bounds for $P_{\mu}$ and
$Q_\mu$. 

\begin{remark}
In the following, except explicitly stated otherwise, 
we will assume that in all the expressions involving $P_{\mu}(x,y)$ or $Q_{\mu}(x,y)$
we have $x\ge 0$, $y>0$ and $\mu>0$; $\mu =0$ will be also allowed for $Q_{\mu}(x,y)$. 
\end{remark}

First, we observe that $P_{\mu +1}(x,y)/P_{\mu }(x,y)<Q_{\mu+1}(x,y)/Q_{\mu}(x,y)$ because, using 
(\ref{eq:PQcompl}), this is equivalent to $Q_{\mu}(x,y)-Q_{\mu+1}(x,y)<0$ which holds on account of
the recurrence (\ref{RecQ}).

We will prove that both ratios are increasing as a function of $y$ and decreasing as functions of $x$ and
$\mu$. Also, we derive upper and lower bounds for these ratios.

\subsection{The ratio $P_{\mu+1}/P_{\mu}$}

\begin{theorem}
\label{cotaPse}
The following bound holds
\begin{equation}
P_{\mu +1}(x,y)/P_{\mu }(x,y)<c_{\mu +1}(x,y).
\end{equation}
\end{theorem} 
\begin{proof}
Using the series (\ref{seriey}) we have
\begin{equation}
\label{cociente}
\Frac{P_{\mu +1}(x,y)}{P_{\mu }(x,y)}=c_{\mu +1}(x,y)
\Frac{1+\displaystyle\sum_{j=1}^{\infty} \left(\prod_{i=1}^{j} c_{\mu+i+1}\right)}
{1+\displaystyle\sum_{j=1}^{\infty} \left(\prod_{i=1}^{j} c_{\mu+i}\right)},
\end{equation}
but $c_{\nu}>c_{\nu+1}$ for $\nu\ge 0$ (Lemma \ref{coefi}) and then Eq. (\ref{cociente}) proves the result.
\end{proof}

Because $P_{\mu+1}(x,y)<P_{\mu}(x,y)$ (see (\ref{RecQ})) we also have that
\begin{corollary}
\label{cotaPP}
\begin{equation}
\Frac{P_{\mu +1}(x,y)}{P_{\mu }(x,y)}<\min\{1,c_{\mu +1}(x,y)\}.
\end{equation}
\end{corollary}

\begin{remark}
The bound $P_{\mu +1}(x,y)/P_{\mu }(x,y)<c_{\mu +1}(x,y)$ is sharp as $x\rightarrow +\infty$, 
$y\rightarrow 0$ and as $\mu\rightarrow +\infty$.
Indeed, Eq. (\ref{cociente}) shows (considering the limiting forms 10.30.4, 10.30.1, 10.41.1 of \cite{Olver:2010:BF}) 
that $h_{\mu}(x,y)=P_{\mu +1}(x,y)/P_{\mu }(x,y)$
is such that
$$
\lim_{x\rightarrow +\infty}\Frac{1}{c_{\mu +1}(x,y)}h_{\mu}(x,y)=
\lim_{y\rightarrow 0}\Frac{1}{c_{\mu +1}(x,y)}h_{\mu}(x,y)=
\lim_{\mu\rightarrow +\infty}\Frac{1}{c_{\mu +1}(x,y)}h_{\mu}(x,y)=1 .
$$

The trivial bound $P_{\mu +1}(x,y)/P_{\mu}(x,y)<1$ is sharp as $y\rightarrow +\infty$.
\end{remark}

\begin{theorem}
\label{cociP}
The function $h_{\mu}(x,y)=P_{\mu+1}(x,y)/P_{\mu}(x,y)$ is increasing as a function of
$y$ and decreasing as a function of $x$ and $\mu$.
\end{theorem}
\begin{proof}
Considering Eqs. (\ref{TTRR}) and (\ref{de1}) for the $P$ function we 
obtain
\begin{equation}
\label{ricH}
\Frac{\partial h_{\mu}}{\partial y}
=\Frac{1}{c_{\mu}}(h_\mu -1)(h_{\mu}-c_{\mu}),
\end{equation}
and Corollary \ref{cotaPP} proves that $\partial h_{\mu}(x,y)/\partial y >0$ because $h_{\mu}<1$ and $h_{\mu}<c_{\mu+1}<c_{\mu}$.

Proceeding similarly with respect to $x$, we have, using (\ref{TTRR}) and (\ref{de2}),
$$
\Frac{\partial h_{\mu}}{\partial x}=-(h_{\mu}-1)(h_{\mu}-c_{\mu+1})<0,
$$
which implies that $\partial h_{\mu}/\partial x<0$.

With regard to the monotonicity with respect to $\mu$, the proof is similar to the proof of the monotonicity 
of $c_{\mu}$ with respect
to $\mu$ (Lemma \ref{coefi}, see in particular the footnote). Indeed, because for 
small $y$ we have $h_{\mu}(x,y)\approx c_{\mu+1}(x,y)$ then $h_{\mu}(x,0)<h_{\mu'}(x,0)$ if $\mu>\mu'$ and in addition, for fixed
$h_{\mu}$ and taking the derivative of (\ref{ricH}), we have
$$
\Frac{\partial}{\partial \mu}h_{\mu}'=h_\mu (h_\mu -1) \Frac{\partial}{\partial \mu}c_{\mu}^{-1}<0
$$
(because $0<h_\mu<1$ and $\partial c_\mu /\partial \mu<0$). These two facts are enough to prove that $h_\mu$ is monotonically 
 decreasing with respect to $\mu$.

\end{proof}

An immediate consequence of the monotonicity property with respect to $\mu$ is the following Tur\'an-type
inequality
\begin{corollary}
\label{TurP}
\begin{equation}
\label{PP}
P_{\mu +1}(x,y)^2-P_{\mu}(x,y)P_{\mu+2}(x,y)>0.
\end{equation}
\end{corollary}
The same inequality will hold for the $Q$-function and for the same reason.

\subsubsection{Convergent bounds for $P_{\mu +1}(x,y)/P_{\mu }(x,y)$}

We observe that the argument used in the proof of Theorem \ref{cotaPse} can be used to obtain sharper bounds, as we next
show.

\begin{theorem}
\label{cociente2}
\begin{equation}
\Frac{\displaystyle\sum_{k=0}^n F_{\mu+k+1}}{\displaystyle\sum_{k=0}^{n+1} 
F_{\mu+k}}=l_{\mu}^{(n)}(x,y)<\Frac{P_{\mu +1}(x,y)}{P_{\mu }(x,y)}<u_{\mu }^{(n)}(x,y)=
\Frac{\displaystyle\sum_{k=0}^n F_{\mu+k+1}}{\displaystyle\sum_{k=0}^n F_{\mu+k}}, 
\end{equation}
for any $n\ge 0$.

These are bounds converging to $\Frac{P_{\mu +1}(x,y)}{P_{\mu }(x,y)}$ as $n\rightarrow +\infty$.
The bounds 
are sharp as $x\rightarrow +\infty$, $y\rightarrow 0$ and $\mu\rightarrow +\infty$.

\end{theorem}

\begin{proof} 
Considering (\ref{cociente}) we have
$$
\Frac{P_{\mu +1}(x,y)}{P_{\mu }(x,y)}=c_{\mu +1}(x,y)
\Frac{A_n+1+\displaystyle\sum_{j=1}^{n} \left(\prod_{i=1}^{j} c_{\mu+i+1}\right)}
{B_n+1+\displaystyle\sum_{j=1}^{n} \left(\prod_{i=1}^{j} c_{\mu+i}\right)},
$$
where 
$$A_n=\displaystyle\sum_{j=n+1}^{\infty} \left(\prod_{i=1}^{j} c_{\mu+i+1}\right)<B_n=\displaystyle\sum_{j=n+1}^{\infty} \left(\prod_{i=1}^{j} c_{\mu+i}\right)$$
because $c_\nu$ is decreasing as a function of $\nu$. Therefore
$$
\Frac{P_{\mu +1}(x,y)}{P_{\mu }(x,y)}<c_{\mu +1}(x,y)
\Frac{1+\displaystyle\sum_{j=1}^{n} \left(\prod_{i=1}^{j} c_{\mu+i+1}\right)}
{1+\displaystyle\sum_{j=1}^{n} \left(\prod_{i=1}^{j} c_{\mu+i}\right)}
$$
which gives the upper bound $u_{\mu }^{(n)}(x,y)$.

On the other hand, from (\ref{sumaPs}) we have, denoting $S=\displaystyle\sum_{k=n+1}^{\infty} F_{\mu+k+1}$,
$$
\Frac{P_{\mu +1}(x,y)}{P_{\mu }(x,y)}=\Frac{S+\displaystyle\sum_{k=0}^n F_{\mu+k+1}}{S+\displaystyle\sum_{k=0}^{n+1}
F_{\mu+k}}>\Frac{\displaystyle\sum_{k=0}^n F_{\mu+k+1}}{\displaystyle\sum_{k=0}^{n+1}
F_{\mu+k}}.
$$

The convergence is immediate due to Eq. (\ref{cociente}) and the sharpness is also straightforward to verify. 
\end{proof}

\subsection{The ratio $Q_{\mu+1}/Q_{\mu}$}

For the $Q$ function, the equations for the derivatives and the monotonicity properties remain the same. But the bounds 
are different; we obtain these bounds as a consequence of monotonicity properties.

\begin{theorem}
\label{cococo}
The function $H_{\mu}(x,y)=Q_{\mu+1}(x,y)/Q_{\mu}(x,y)$ is increasing as a function of
$y$ and decreasing as a function of $x$ and $\mu$. In addition
\begin{equation} 
\label{boo1}
H_{\mu}(x,y)>\max\{1,c_{\mu}(x,y)\}.
\end{equation}
\end{theorem}
\begin{proof}

We prove the monotonicity with respect to $y$ and the bounds; the rest of properties are proved in the
same way as we did for the $P$-function (theorem \ref{cociP}).

Same as for the $P$ function, the ratio
$H_{\mu}(x,y)=Q_{\mu+1}(x,y)/Q_{\mu}(x,y)$ satisfies the Riccati equation
\begin{equation}
\label{ricHQ}
\Frac{\partial H_{\mu}(x,y)}{\partial y}=\Frac{1}{c_{\mu}}(H_\mu -1)(H_{\mu}-c_{\mu})
\end{equation}
and we would have $H'_{\mu}=0$, for $H_\mu$ equal to the characteristic roots $1$ and $c_{\mu}(x,y)$. 
Because $c_{\mu}(x,0)=0$ and $c_{\mu}(x,+\infty)=+\infty$ and $c_\mu$ is increasing as a function of $y$ 
(Lemma \ref{coefi}) there exists
a value $y_0$ such that $c_{\mu}(x,y_0)=1$ and the curves $H_{\mu}=1$ and $H_{\mu}=c_\mu (x,y)$ divide
the $(y,H_{\mu})$-plane in four different zones. We prove now that
the function of $y$ $H_{\mu}(x,y)=Q_{\mu}(x,y)/Q_{\mu-1}(x,y)$ lies in the region $H_{\mu}(x,y)>\max\{1,c_{\mu}(x,y)\}$. 

From (\ref{RecQ}) we have that $H_{\mu}(x,y)>1$. For $y>y_0$ (where $c_{\mu}(x,y)>1$) two different situations 
may occur: either $H_{\mu}(x,y)>c_{\mu}(x,y)$ for all $y>y_0$ or there exists
$y_1>y_0$ such that $H_\mu(x,y_1)=c_{\mu}(x,y_1)$. But this last situation can not occur, as we next prove, which shows that  
$H_{\mu}(x,y)=c_{\mu}(x,y)$ for all $y>y_0$ and then the bound (\ref{boo1}) holds.

Observe that if such a value $y_1$ existed, because $c_\mu$ is increasing as a function of $y$, 
for $y>y_0$ the graph of $H_\mu$ would enter the region $1<H_{\mu}(x,y)<c_{\mu}(x,y_1)$,
 where $\partial H_{\mu}(x,y)/\partial y <0$; but then
we would have $1<H_{\mu}(x,y)<c_{\mu}(x, y_1)$ for all $y>y_1$. But this is not the case, because $H_{\mu}(x,y)=Q_{\mu+1}(x,y)/Q_{\mu}(x,y)$ is unbounded as $y\rightarrow +\infty$. Indeed, with the definition (\ref{eq:defQmu}) and using
L'H\^opital's rule (recall that $Q_{\mu}(x,+\infty)=0$)
we have
$$
\lim_{y\rightarrow +\infty}\Frac{Q_{\mu+1}(x,y)}{Q_{\mu}(x,y)}=\lim_{y\rightarrow +\infty}c_{\mu}(x,y)=+\infty .
$$

We observe that the inequality $H_{\mu}(x,y)>\max\{1,c_{\mu}(x,y)\}$ together with (\ref{ricHQ}) 
proves that  $H_{\mu}(x,y)$ is increasing as a function of $y$.

\end{proof}

\begin{remark}
The bound $H_{\mu}(x,y)>c_{\mu}(x,y)$ is sharp for large $y$ because using asymptotic information \cite{Temme:1993:ANA} 
it is possible to check that
\begin{equation}
\lim_{y\rightarrow +\infty}\Frac{H_{\mu}(x,y)}{c_{\mu}(x,y)}=1.
\end{equation}
On the other hand, the bound $H_{\mu}(x,y)>1$ is sharp as $\mu\rightarrow +\infty$ (see (\ref{limimu})).

\end{remark}

\subsubsection{Bounds from the continued fraction}

From the recurrence relation (\ref{TTRR}) we can write
\begin{equation}
\label{CF1}
\Frac{P_{\mu +1}(x,y)}{P_{\mu }(x,y)}=\Frac{c_{\mu +1}}{1+c_{\mu +1}-\Frac{P_{\mu +2}(x,y)}{P_{\mu +1}(x,y)}}
\end{equation}
and because $P_{\mu}$ is a minimal solution of the three-term recurrence relation, iterating (\ref{CF1})
we obtain a convergent continued fraction. The successive approximants form
a convergent sequence of lower bounds and
\begin{equation}
\label{cfP2}
\Frac{P_{\mu +1}(x ,y)}{P_{\mu }(x,y )}>\Frac{c_{\mu +1}}{1+c_{\mu+1}\,-}\,\Frac{c_{\mu+2}}{1+c_{\mu+2}\,-}\,\ldots\,
\Frac{c_{\mu+k+1}}{1+c_{\mu+k+1}},
\end{equation}
which is the same as the lower bound in Theorem \ref{cociente2}.
In particular, with the first approximant we get
\begin{equation}
\label{cfP3}
\Frac{P_{\mu +1}(x ,y)}{P_{\mu }(x,y )}>\Frac{c_{\mu +1}}{1+c_{\mu+1}}.
\end{equation}

If the tail of the CF is substituted by an upper bound we obtain an upper bound. Using Theorem \ref{cotaPse}, we get
\begin{equation}
\label{cfP3}
\Frac{P_{\mu +1}(x ,y)}{P_{\mu}(x,y )}<\Frac{c_{\mu +1}}{1+c_{\mu +1}\,-}\,\Frac{c_{\mu+2}}{1+c_{\mu+2}\,-}\,\ldots\,
\Frac{c_{\mu+k +1}}{1+c_{\mu+k +1}-c_{\mu+k+2}},
\end{equation}
and in particular
\begin{equation}
\label{cfP4}
\Frac{P_{\mu +1}(x ,y)}{P_{\mu}(x,y )}<\Frac{c_{\mu +1}}{1+c_{\mu +1} -c_{\mu+2}}.
\end{equation}
For $Q_\mu$, because it is dominant, the application of the recurrence in the forward direction gives better
bounds. We write
\begin{equation}
\Frac{Q_{\mu+1}}{Q_\mu}=1+c_{\mu}-c_{\mu}\Frac{1}{\Frac{Q_{\mu}}{Q_{\mu-1}}}.
\end{equation}
From this we see that, because $c_{\mu}Q_{\mu-1}/Q_{\mu}>0$ if $\mu\ge 1$
\begin{equation}
\label{otrobo}
\Frac{Q_{\mu+1}}{Q_\mu}<1+c_{\mu},\mu\ge 1 .
\end{equation}
We don't consider further iterations.

The inequality (\ref{otrobo}) is valid for $\mu\ge 0$ if $xy\ge 1$ (see Corollary \ref{xyuno}).

\subsection{Bounds for the functions $P_\mu$ and $Q_\mu$}

Considering the inhomogeneous recurrence relations (\ref{RecQ}), the previous bounds on ratios can be
translated into bounds for the functions themselves. 

\subsubsection{The $P$-function}

We start with the function $P_\mu$, and we consider a bound $r_\mu$ such that
\begin{equation}
\Frac{P_{\mu+1}}{P_{\mu}}<r_{\mu}.
\end{equation}
For a lower bound all the subsequent inequalities will be reversed. In both cases, 
we assume that $r_\mu<1$.

Writing the recurrence relation (\ref{RecQ}) as $P_{\mu+1}=P_\mu-F_\mu$, 
 we have $P_\mu-F_\mu<r_\mu P_\mu$ and then
\begin{equation}
\label{unoco}
P_\mu < \Frac{1}{1-r_\mu}F_{\mu}=b^{(1)}_\mu,
\end{equation}
 where $F_\mu$ given by (\ref{Fmu}). On the other hand $P_{\mu+1}=P_\mu-F_\mu>P_{\mu+1}/r_\mu-F_\mu$ and therefore
\begin{equation}
P_{\mu}<r_{\mu-1} b^{(1)}_{\mu -1} =b^{(2)}_\mu.
\end{equation}
Considering  $r_{\mu}=c_{\mu+1}$ (Theorem \ref{cotaPse})
we have  $b^{(2)}_\mu/b^{(1)}_\mu=(1-c_{\mu+1})(1-c_{\mu})>1$
and the second bound is worse. The first bound gives
\begin{equation}
\label{pp2}
P_\mu <\Frac{1}{1-c_{\mu+1}}F_\mu ,
\end{equation}
valid when $c_{\mu+1}<1$.

Considering now
the lower bound with $r_\mu=c_{\mu +1}/(1+c_{\mu +1})$ (Eq. (\ref{cfP3})) we have 
$b^{(2)}_\mu/b^{(1)}_\mu=(1+c_{\mu +1})^{-1}<1$ and because we are dealing with lower bounds, the first
bound is preferable, which is
\begin{equation}
P_{\mu}>(1+c_{\mu +1})F_\mu .
\end{equation}

\subsubsection{The $Q$-function}We consider 
$$
\Frac{Q_{\mu+1}}{Q_\mu}>r_\mu
$$
and write the recurrence as $Q_{\mu+1}=Q_\mu+F_\mu$. Then $Q_\mu + F_\mu > r_\mu Q_\mu$ and
\begin{equation}
Q_\mu<\Frac{1}{r_\mu -1}F_\mu =B^{(1)}_\mu ,
\end{equation}
provided $r_\mu >1$. On the other hand $Q_{\mu+1}<Q_{\mu+1}/r_\mu +F_\mu$ and therefore
\begin{equation}
Q_{\mu}<r_{\mu-1}B^{(1)}_{\mu-1}=B^{(2)}_\mu
\end{equation}

Taking $r_\mu=c_\mu$,
$$
\Frac{B^{(2)}}{B^{(1)}}=\Frac{1-c_{\mu}^{-1}}{1-c_{\mu-1}^{-1}}<1.
$$

Therefore the first bound is superior when it holds. This gives
\begin{equation}
\label{qq1}
Q_{\mu}<\Frac{1}{1-c_{\mu-1}^{-1}}F_{\mu-1},\mu\ge 1 ,
\end{equation}
valid if $c_{\mu-1}>1$.
The second bound, valid if $c_{\mu}>1$, is
\begin{equation}
\label{qq1}
Q_{\mu}<\Frac{1}{c_{\mu}-1}F_{\mu},\mu\ge 0 .
\end{equation}

Lower bounds can be obtained from (\ref{otrobo}). In this case
$$
\Frac{B^{(2)}_\mu}{B^{(1)}_\mu}=1+c_\mu^{-1}>1
$$
and then the second bound will be sharper although the validity will be more restricted.
The first bound is
\begin{equation}
\label{eqiq1}
Q_{\mu}>c_{\mu}^{-1}F_\mu =F_{\mu -1},\mu\ge1
\end{equation}
and the second
\begin{equation}
\label{eqiq2}
Q_{\mu}>(1+c_{\mu -1}^{-1})F_{\mu-1},\mu\ge2
\end{equation}

If $xy\ge 1$ the validity of the last two inequalities can be extended by considering Corollary \ref{xyuno}. 
Indeed, using the recurrence (\ref{RecQ}), $Q_{\mu}(x,y)=F_{\mu-1}(x,y)+Q_{\mu-1}(x,y)>F_{\mu-1}(x,y)$ (which is
Eq. (\ref{eqiq1})) if
$\mu\ge -1$ and $Q_{\mu}(x,y)=F_{\mu-1}(x,y)+F_{\mu-2}(x,y)+Q_{\mu-2}(x,y)>F_{\mu-1}(x,y)+F_{\mu-2}(x,y)$ (Eq. 
(\ref{eqiq2})) if $\mu\ge 0$.

\subsubsection{Combining the bounds}

Some of the previous bounds had a limited range of validity, depending on the value of $c_\mu$ (smaller or larger than $1$); 
an explicit region of validity in terms of $x$, $y$ and $\mu$ can be given 
using Lemma \ref{coefi}. Combining the results for the $P$ and $Q$ functions we can write the following theorem summarizing such bounds.

\begin{theorem}
\label{mesbobo}
Let $F_{\mu}(x,y)$ be the probability density function such that $Q_{\mu+1}(x,y)=\int_y^{+\infty} F_{\mu}(x,t)dt$ and
$c_{\mu}(x,y)=F_{\mu}(x,y)/F_{\mu-1}(x,y)$ then the following bounds hold:

\begin{equation}
\label{mes1}
\mkern-37mu 1.\quad Q_{\mu}>1-\Frac{F_\mu}{1-c_{\mu+1}},\,\mu>0,\,y<x+\mu+1/2.
\end{equation}

\begin{equation}
\label{mes2}
\mkern-135mu 2.\quad Q_{\mu}<\Frac{F_\mu}{c_{\mu}-1},\, \mu\ge 0,\, y>x+\mu.
\end{equation}

\begin{equation}
\label{mes3}
\mkern-83mu 3.\quad Q_{\mu}<\Frac{F_{\mu-1}}{1-c_{\mu-1}^{-1}},\, \mu\ge 1,\, y>x+\mu -1.
\end{equation}

\begin{equation}
\label{mas1}
\,\mkern-20mu 4.\quad Q_\mu (x,y)<1-(1+c_{\mu +1}(x,y))F_{\mu}(x,y) ,\mu>0.
\end{equation}

\begin{equation}
\label{mas2}
 \mkern-172mu \mkern-18mu 5.\quad Q_\mu (x,y)>F_{\mu-1},\,\mu\ge 1.
\end{equation}

\begin{equation}
\label{mas3}
\mkern-112mu 6.\quad Q_\mu (x,y)>(1+c_{\mu -1}^{-1})F_{\mu -1},\mu\ge 2.
\end{equation}

\vspace*{0.5cm}
The bound (\ref{mes3}) is superior to (\ref{mes2}) when it holds. 

The bound (\ref{mas3}) is superior to (\ref{mas2}) when it holds. If $xy\ge 1$ (\ref{mas2}) 
also holds for $\mu\ge -1$ and (\ref{mas3}) for $\mu\ge 0$.
\end{theorem}

As we will see, (\ref{mes1}) and (\ref{mas1}) become sharp as we move away from the transition line $y=x+\mu$
with $y<x+\mu$ (when $P$ is smaller than $Q$), while the rest of bounds become sharper away from the transition line $y=x+\mu$
but with $y>x+\mu$ (when $Q$ is smaller than $P$).

We have only considered the most simple bounds involving two Bessel functions at most (or a Bessel function and a ratio of Bessel functions, which is easy to compute with a continued fraction), but improvements can be obtained
particularly in the case of the $P$-function, for which convergent sequences of bounds are available. For instance, considering
Eq. (\ref{cfP4}) we obtain the following bound
\begin{equation}
\label{betterlo}
P_{\mu}(x,y)<\left(1+\Frac{c_{\mu+1}}{1-c_{\mu+2}}\right)F_{\mu}(x,y),\, y<x+\mu+3/2,
\end{equation}
and it is easy to show that it is sharper than the first bound of Theorem \ref{mesbobo}.\footnote{Although three Bessel functions
appear in this bound, it is possible to write down the bound in terms of two functions by using the three-term recurrence relation for Bessel
functions}

In the next section we briefly describe some other types of convergent bounds for the $P$-function.

\subsubsection{Convergent sequences of bounds}

Considering (\ref{sumaPs}), sequences of bounds can be obtained as follows.
    
\begin{proposition}
\label{convergence}
If $B_{\mu}(x,y)$ is an upper (lower) bound for $P_{\mu}(x,y)$ then the following are upper (lower) bounds for
any $n$
\begin{equation}
\label{bobo}
B^{(n)}_{\mu}(x,y)=B_{\mu +n +1}(x,y)+e^{-x-y}\displaystyle\sum_{k=0}^{n}\left(\Frac{y}{x}\right)^{\frac{\mu+k}{2}} I_{\mu+k}(2\sqrt{xy}).
\end{equation}
If the bound $B_{\mu}(x,y)$ is such that $B_{\mu}(x,y)\rightarrow 0$ as $\mu\rightarrow +\infty$ then 
$\{B^{(n)}_{\mu}(x,y)\}$ is a sequence of upper (lower) bounds converging to $P_{\mu}(x,y)$ as $n\rightarrow \infty$.
\end{proposition}

We observe that the bounds (\ref{mes1}) and (\ref{mas1})  of Theorem \ref{mesbobo} will give 
convergent sequences when Proposition \ref{convergence} is considered.

On the other hand, we already gave convergent sequences of upper bounds for $P_{\mu +1}/P_{\mu }$ and taking
$r_\mu=B_{\mu}^{(n)}$ in (\ref{unoco}) with $B_{\mu}^{(n)}$ given by $u_{\mu}^{(n)}$ of Theorem \ref{cociente2} we have:

\begin{proposition}
\label{superior}
For $n=0,1,\ldots$ we have
$$
P_{\mu}(x,y)<U_{\mu}^{(n)}(x,y)=\Frac{\displaystyle\sum_{k=0}^n F_{\mu +k}}{F_\mu - F_{\mu+n+1}}F_{\mu}
$$
if $F_\mu(x,y)> F_{\mu+n+1}(x,y)$. The sequence of bounds is convergent as $n\rightarrow +\infty$.
\end{proposition}

 Using the bound of
Theorem \ref{superior} and applying Theorem (\ref{convergence}) for this bound and for $B_{\mu +n +1}(x,y)=0$ we have the
following result:
\begin{proposition}
$$
\sum_{k=0}^{n}F_{\mu +k} <P_{\mu}(x,y)<
\sum_{k=0}^{n} F_{\mu+k} +U_{\mu+n+1}^{(q)}(x,y),
$$
where the upper bound is valid provided that  $F_{\mu+n+1}(x,y)>F_{\mu+n+q+2}(x,y)$.
\end{proposition}
The lower bound of this theorem can also be obtained taking
$r_\mu=B_{\mu}^{(n)}$ in (\ref{unoco}) with $B_{\mu}^{(n)}$ given by $l_{\mu}^{(n)}$ of Theorem \ref{cociente2}.

Other sequences of convergent bounds can be obtained from the expression of Marcum functions in
series of incomplete gamma functions (see, for instance, Eq. (7) of \cite{Gil:2014:COT}). Truncating the series we have

\begin{equation}\label{qratio}
e^{-x}\sum_{k=0}^{n} \frac{x^k}{k!}  P_{\mu +k}(y) <P_{\mu} (x,y)<1-e^{-x}\sum_{k=0}^{n} \frac{x^k}{k!}  Q_{\mu +k}(y)
\end{equation}
and these bounds also converge to $P_{\mu} (x,y)$ as $n\rightarrow +\infty$. In \cite{Paris:2013:SBF}, the bounds
for the particular case $n=2$ are analyzed. After using the first order non-homogeneous recurrence satisfied by the
incomplete gamma function ratios (Eq. (\ref{RecQ}) in the limit $x\rightarrow 0$)

\section{Central distributions}

\label{central}

Taking the limit $x\rightarrow 0$ we obtain as a consequence properties for the incomplete gamma function
ratios. Considering that $c_{\mu}(0^+,y)=y/\mu$ and $F_{\mu}(0^+,y)=y^{\mu}e^{-y}/\Gamma (\mu+1)$
a number of particular results for the incomplete gamma functions follow from the results for the
noncentral distribution.

For instance, from the monotonicity properties with respect to $\mu$ of Theorems \ref{cociP} and \ref{cococo}
we have, taking $a=\mu$ and $x=0$:
\begin{corollary}
The ratios 
$
\Frac{1}{a-1}\Frac{\gamma (a,y)}{\gamma (a-1,y)}$ and $\Frac{1}{a-1}\Frac{\Gamma (a,y)}{\Gamma (a-1,y)}
$
are decreasing as a function of $a>1$.
\end{corollary}

We next find related functions with are increasing as a function of $a$. 
This will lead to Tur\'an-type inequalities which can be used for obtaining further bounds.

\begin{theorem}
\label{papapa}
The function
\begin{equation}
p_a (y)=\Frac{a}{a-1}\Frac{\gamma (a,y)}{\gamma (a-1,y)}
\end{equation}
is increasing as a function of $a$ for $a>1$, $y>0$.
\end{theorem}
{\it {\bf Proof}
Considering (\ref{ricH}) we obtain the following Riccati equation for $p_a (y)$,
where the derivative is taken with respect to $y$
\begin{equation}
\label{RiccatinR}
p_a^{\prime}(y)=a-\left(1+\frac{a-1}{y}\right) p_a (y) +\frac{1}{y}\left(1-\frac{1}{a}\right)p_a^2 (y).
\end{equation}

Using the power series \cite[Eq. 8.7.1]{Paris:2010:IGR} we have that
\begin{equation}
p_a (y)=y-\Frac{1}{a(a+1)}y^2 +{\cal O}(y^3),
\end{equation}
therefore we have $p_{a+\epsilon}(y)>p_{a}(y)$ for $y$ sufficiently close to $y=0$; but then it is easy
to see that this must hold for any $y>0$. Too see this, let us assume $p_{a+\epsilon}(y_e)=p_{a}(y_e)$ and 
$p_{a+\epsilon}(y)>p_{a}(y)$ for $y\in (0,y_e)$ and we will arrive at a contradiction. Indeed, this immediately 
implies that $p_{a+\epsilon}^{\prime}(y_e)<p_{a}^{\prime}(y_e)$ and with $p_{a+\epsilon}(y_e)=p_{a}(y_e)$  but this is 
in contradiction with the equation (\ref{RiccatinR}), because for a same value of $p_{a}$ (fixed) the derivative $p_a^{\prime}$
increases with $a$. This can be checked by taking the partial derivative of $p_{a}'$ with respect to $a$ with $y$ and $p_a$ fixed:
$$
\Frac{\partial p_a^{\prime}}{\partial a}=1-\Frac{1}{y}p_a+\Frac{1}{y a^2}p_a^2>1-\Frac{1}{y}p_a>0,
$$
where the last inequality is true because 
\begin{equation}
\label{previbo}
\Frac{\gamma (a,y)}{\gamma (a-1,y)}<\left(\Frac{a-1}{a}\right)y,\,a>1,y>0,
\end{equation}
which is a consequence of Theorem \ref{cotaPse} (taking $x\rightarrow 0$).
}

\begin{remark}
In \cite{Qi:2002:MRI} it was proved that $\gamma (a,y)/\gamma (a-1,y)$ is increasing as a function of $a$. The 
result of Theorem \ref{papapa} is an improvement.
\end{remark}

This monotonicity property can be used to obtain a Tur\'an type inequality. We combine this with the inequality of
Corollary \ref{TurP} and we get:

\begin{corollary}
\label{Turga0}
\begin{equation}
1-\Frac{1}{a}<\Frac{\gamma (a,y)^2}{\gamma(a+1,y)\gamma (a-1,y)}<1-\Frac{1}{a^2},\, a>1.
\end{equation}
\end{corollary}

The  inequalities in Corollary \ref{Turga0} are not new, see \cite{Merkle:1993:SIF}.
The right inequality can be used to get an upper bound, as was done in \cite{Merkle:1993:SIF}. We do this in a
slightly different way and get an improvement; similar analysis for the lower bound gives Theorem \ref{cotaPse} for $x=0$.

We consider the upper bound in Corollary \ref{Turga0} and eliminate $\gamma (a+1,y)$ 
by using the recurrence \cite[Eq. 8.8.2]{Paris:2010:IGR}
\begin{equation}
\label{pa2}
\gamma (a+1,y)-(a+y)\gamma (a,y) +(a-1) y \gamma (a-1,y)=0.
\end{equation}
This gives
\begin{equation}
h_a (y)^2-b_a(a+y) h_a (y) +b_a(a-1)y<0
\end{equation}
where 
\begin{equation}
h_a (y)=\gamma(a,y)/\gamma (a-1,y),\,b_a=1-\Frac{1}{a^2}.
\end{equation}
Two bounds are obtained from this inequality:
\begin{equation}
h_a(y)<\Frac{b_a}{2}(y+a+\sqrt{(y-a)^2+4ay/(a+1)})
\end{equation}
and
\begin{equation}
\label{eqq2}
h_a (y)>\Frac{2(a-1)y}{y+a+\sqrt{(y-a)^2+4a y/(a+1)}}=L_h (a,y).
\end{equation}

We proceed with the last inequality. Using
\begin{equation}
\gamma (a+1,y)=a\gamma (a,y)-y^a e^{-y}
\end{equation}
and using that $\gamma (a+1,y)>L_h (a,y) \gamma (a,y)$ for $a>0$
we get
\begin{equation}
\label{unacota}
\begin{array}{l}
\gamma (a,y)>b_2 (a)=\Frac{1}{2r_a a}y^a e^{-y}\left[2r_a+L_a+\sqrt{L_a^2+4r_ay}\right],\\
\\
r=(a+1)/(a+2),\,L=y-a-1,\,a>0.
\end{array}
\end{equation}

Using the first order inhomogeneous relation, Merkle \cite[Eq. (28)]{Merkle:1993:SIF} 
obtained a related bound, which can be written as
\begin{equation}
\gamma (a,y)>b_1 (a)=\Frac{1}{2r_{a-1}}y^{a-1} e^{-y}\left[L_{a-1}+\sqrt{L_{a-1}^2+4r_{a-1}y}\right],\,a>1.
\end{equation}

Both bounds are related by $b_2 (a)=(y^a e^{-y} + b_1 (a+1))/a$; in other words the bound (\ref{unacota}) can be
obtained considering Merkle's bound together with the relation 
\begin{equation}
\label{inhoga}
\gamma (a,y)=\Frac{1}{a}\left[y^a e^{-y}+\gamma (a+1,y)\right],
\end{equation}
$b_2$
turns out to be a sharper bound that Merkle's bound. 

It is clear that if in (\ref{inhoga}) $\gamma (a+1,y)$ is substituted by a lower (upper) bound, this gives a
lower (upper) bound for $\gamma (a,y)$. We can consider additional iterations with this inhomogeneous
recurrence. After $n$ applications of the backward recurrence we arrive to a result corresponding to Theorem \ref{convergence}
in the limit $x\rightarrow 0$.

Regarding $H_a(y)$, it is known that this ratio is monotonically increasing as a function of $a>1$, as
was proved in \cite{Qi:2002:MRI}; it is also monotonic as a 
function of $y$ (see Theorem \ref{cococo}, 
which also proves that $H_a (y)/(a-1)$ is decreasing as function of $a>1$). We give an alternative (and
elementary) proof of the fact that $H_a(y)$ is increasing and that it is in fact increasing for all real $a$ and not
only for $a>1$. First we prove the following:
\begin{proposition}
$\Frac{\Gamma (a,y)}{\Gamma (a-1,y)}>y$ for all $a\in {\mathbb R}$ and $y>0$.
\end{proposition}
\begin{proof}
Let $G(y)=\Gamma (a,y)-y\Gamma (a-1,y)$, then $G'(y)=-\Gamma (a-1,y)<0$ and using the asymptotic
expansion \cite[Eq. 8.11.2]{Paris:2010:IGR} we see that, for large $y$, $G(y)=y^{a-2}e^{-y}(1+{\cal O}(y^{-1}))$.
Therefore $G(y)>0$ for sufficiently large $y$ and $G'(y)<0$ for all $y>0$, which proves that $G(y)>0$ for all $y>0$.
\end{proof}

\begin{theorem}
\label{Gainc}
$H_a(y)=\Gamma (a,y)/\Gamma (a-1,y)$ is increasing as a function of $a\in {\mathbb R}$.
\end{theorem}
\begin{proof}
First we notice that for large enough $y$, $H_{a+\epsilon}(y)>H_a (y)$; indeed, using the asymptotic
expansion \cite[Eq. 8.11.2]{Paris:2010:IGR}  we see that
$$
\Frac{\partial H_a (y)}{\partial a}=\Frac{1}{y}(1+{\cal O}(y^{-1})).
$$

Considering the Riccati equation satisfied by $H_a (y)$
\begin{equation}
H_a^{\prime} (y)=(a-1)-\left(1+\frac{a-1}{y}\right) H_a (y) +\frac{1}{y}H_a^2 (y)
\end{equation}
where the derivative is taken with respect to $y$. Then, for a fixed value of $H_a (y)$ (not
depending on $a$)
\begin{equation}
\Frac{\partial H_a^{\prime} (y)}{\partial a}=1-\frac{1}{y} H_a (y)<0,
\end{equation}
where in the last inequality we have used that $H_a (y)>y$ (see Theorem (\ref{Gainc})). Now, because  
$H_{a+\epsilon}(y_+)>H_a (y_+)$ for large enough $y_+$, using a similar reasoning as in the proof
of Theorem \ref{papapa} if follows that $H_{a+\epsilon}(y)>H_a (y)$
for all positive $y$, which
completes the proof.
\end{proof}

\begin{corollary}
\label{Turga}
\begin{equation}
\begin{array}{l}
\Frac{\Gamma (a,y)^2}{\Gamma(a+1,y)\Gamma (a-1,y)}<1, a\in {\mathbb R},\\
\Frac{a-1}{a}<\Frac{\Gamma (a,y)^2}{\Gamma(a+1,y)\Gamma (a-1,y)}, a>1.
\end{array}
\end{equation}
\end{corollary}
\begin{proof}
The upper bound is a consequence of the fact that $H_a(y)$ is increasing as a function of $a$ 
(Theorem \ref{Gainc}) while the lower bound is true because $H_a (y)/(a-1)$ is decreasing as a function of $a$.
\end{proof}

From the upper bound in Corollary \ref{Turga}, we get, using this inequality together with the three term recurrence relation, similarly as we did for
the ratio $h_a (y)=\gamma (a,y)/\gamma (a-1,y)$, a bound for $H_a (y)$ which leads to a lower bound for 
$\Gamma (a,y)$:
\begin{corollary}
For real $a$ and positive $y$ the following holds:
\begin{equation}
H_a (y)<\Frac{1}{2}\left[y+a+\sqrt{(y-a)^2+4y}\right],
\end{equation}
\begin{equation}
\label{boundGAM}
\Gamma (a,y)>\Frac{2y^a e^{-y}}{y+1-a+\sqrt{(y-a-1)^2+4y}}.
\end{equation}
\end{corollary}

From the lower bound of Corollary \ref{Turga} we would obtain a bound slightly weaker than that
of Theorem \ref{cococo}.

Summarizing and considering also some bounds obtained from Theorem \ref{mesbobo}, we have:
\begin{theorem}
\label{resum}
The following bounds hold for $y>0$:
\begin{enumerate}
\item{}$\gamma (a,y)>l_1(a,y)=\Frac{1}{a}\left(1+\Frac{y}{a+1}\right)y^{a} e^{-y}$, $a>0$.
\item{}$\gamma (a,y)>l_2 (a,y)=\Frac{1}{2 a r_a}y^a e^{-y}\left[2 r_a + L_a + \sqrt{L_a^2 +4r_a y}\right]$, $a>0$, with 
$r_a=(a+1)/(a+2)$, $L_a=y-a-1$.
\item{}$\gamma (a,y)<u_1 (a,y)=\Frac{a +1}{a (a+1-y)}y^{a}e^{-y}$, $y<a+1$, $a>0$.
\item{}$\Gamma (a,y)>L_1 (a,y)=\left(1+A\Frac{a-1}{y}\right)y^{a-1}e^{-y}$, $a\ge 1$, with $A=0$ if $a<2$ and $A=1$ otherwise.
\item{}$\Gamma (a,y)>L_2 (a,y)= \left[y+1-a+\sqrt{(y-a-1)^2+4y}\right]^{-1}2y^a e^{-y}$, $a\in {\mathbb R}$.
\item{}$\Gamma (a,y)<U_1 (a,y)=\Frac{1}{y+1-a}y^{a}e^{-y}$, $y>a-1$, $a\ge 1$.
\end{enumerate}
\end{theorem}

The first inequality in this theorem corresponds to Eq. (\ref{mas1}), the second to (\ref{unacota}), the third to (\ref{mes1}),
the fourth to (\ref{mas2}) and (\ref{mas3}), the fifth to (\ref{boundGAM}) and the last one to (\ref{mes3}).

\section{Comparison with previous bounds}
\label{compar}

Next we compare our bounds with previous bounds. As we will see, our bounds for the non-central distribution
are superior to previous bounds in a wide region of parameters, particularly the upper bounds for $Q_{\mu}(x,y)$.
For the central distributions we observe that our combined bounds give sharper approximations than previous bounds 
for all the range of parameters where they apply.

\subsection{Noncentral case}

We compare our bounds against recent bounds by Paris \cite{Paris:2013:SBF}, Baricz \cite{Baricz:2009:NBF} and Baricz $\&$ Sun 
\cite{Baricz:2009:TBF}. In tables 1 and 2, we label the different bounds for the Marcum-Q function as follows.
\begin{enumerate}
\item{Bounds by Baricz $\&$ Sun}: upper bounds {\bf BU1} \cite[Eq. (15)]{Baricz:2009:NBF}; lower bounds {\bf LB1} 
\cite[Eq. (8)]{Baricz:2009:TBF} and {\bf LB2} \cite[Eq. (17)]{Baricz:2009:TBF}. These are bounds in terms of a Bessel function
and two or three error functions.
\item{Bounds by Paris}: upper bounds {\bf UP1} \cite[Eq. (3.2)]{Paris:2013:SBF} and {\bf UP2} \cite[Eq. (3.4)]{Paris:2013:SBF};
lower bounds {\bf LP1} \cite[Eq. (3.1)]{Paris:2013:SBF} and {\bf LP2} \cite[Eq. (3.3)]{Paris:2013:SBF}. These are bounds
depending on one incomplete gamma function, and which are sharp for values of $x$ and/or $y$ not much larger than $1/2$.
\item{Our bounds}:  upper bounds {\bf US1A} (Eq. (\ref{mes3})), {\bf US1B} (Eq. (\ref{mes2})) and {\bf US2} (Eq. (\ref{mas1}));
lower bounds {\bf LS1} (Eq. (\ref{mas3}))  and {\bf LS2} (Eq. (\ref{mes1})). These are bounds depending on a Bessel function and
a ratio of Bessel function (which is easy to compute with a continued fration and for which sharp bounds exist 
\cite{Segura:2011:BRM}).
\end{enumerate}

In tables 1 and 2 we compare all the previous bounds for several values of the parameters. We compare the bounds
with the actual values of functions. The column labelled with 
L1 contains the best lower bound and the column U1 the best upper bound (with respect to $Q$); 
the columns $L_2$ and $U_2$ contain the second best lower and upper bounds respectively. 

Inside parenthesis, an estimation of the relative error with one significant digit is given. When $P$ is smaller that $Q$ 
we give the relative error for $P$ instead of $Q$, because this is more significant. 
When $Q$ is smaller than $P$ the value of $y$ in the table is written in bold font. Also, our bounds are given
in bold font.

As we observe from the results of the tables, the bounds given in Theorem \ref{mesbobo} of this paper are superior to previous bounds
for certain combinations of parameters, particularly not close to the transition line $y=x+\mu$; they are not, however,
superior in all case but tend to improve for larger $\mu$. Particularly, the upper bounds are a considerable 
improvement since they are usually sharper than previous bounds. 

\begin{center}
\begin{tabular}{|l|l|l||l|l|}
\hline
 x=1  & L1 & U1 & L2 & U2    \\
y  & & & & \\
\hline
1 & LP2 ($0.02$)  & {\bf US2} ($0.1$) & {\bf LS2} ($0.1$) & UP2 ($0.3$)   \\
\hline
{\bf 4}  & LB1 ($0.07$) & {\bf US1A} ($0.09$)  & {\bf LS1} ($0.1$) &   UP1  ($0.2$)   \\
\hline
{\bf 16}   & LB1 ($0.02$) & {\bf US1A} ($0.05$)  &  {\bf LS1} ($0.01$) & {\bf US1B} ($0.06$)\\
\hline
\hline
x= 4 &   &  &  &      \\
y & & & &  \\
\hline
1  &  {\bf LS2} ($0.04$) & {\bf US2} ($0.08$)  & LB2 ($0.1$) & UB1 ($0.7$)   \\
\hline
4   &  LB1 ($0.1$)   &  UB1 ($0.7$) & LB2 ($0.2$) & {\bf US2} ($0.8$)   \\
\hline
{\bf 16}   & LB1 ($0.03$)    & {\bf US1A} ($0.04$) & LS1 ($0.3$) & {\bf US1B} ($0.1$)   \\
\hline
\hline
x= 16   & &  & &    \\
y &  & & & \\
\hline
1    & {\bf LS2} ($0.007$) & {\bf US2} ($0.04$) & LB2 ($0.08$) & UB1 ($0.8$)    \\
\hline
4     & {\bf LS2} ($0.04$) & {\bf US2} ($0.2$) & LB2 ($0.05$) & UB1 ($0.9$)  \\
\hline
16    & LB1 ($0.07$) & UB1 ($0.8$) & LB2 ($0.08$) & {\bf US2} ($2$)   \\
\hline
{\bf 32}    & LB1 ($0.02$) & {\bf US1A} ($0.1$) & {\bf LS1} ($0.8$) & {\bf US1B} ($0.2$)\\
\hline
\end{tabular}
\captionof{table}{Best bounds for $\mu=1$ and several values of $x$ and $y$}
\end{center}

\begin{center}
\begin{tabular}{|l|l|l||l|l|}
\hline
 x=1  & L1 & U1 & L2 & U2    \\
y  & & & & \\
\hline
1 & {\bf LS2} ($0.0002$)  & {\bf US2} ($0.003$) & LP2 ($0.0006$) & UP2 ($0.4$)   \\
\hline
\bf 16  & LP2 ($0.02$) & {\bf US2} ($1.5$)  & {\bf LS1} ($0.8$) &  UP2  ($3$)   \\
\hline
{\bf 32}   & LP2 ($0.2$) & {\bf US1A} ($0.05$)  &  {\bf LS1} ($0.3$) &   {\bf US1B} ($0.1$)\\
\hline
{\bf 64}   & {\bf LS1} ($0.08$) & {\bf US1A} ($0.007$)  &  LB1 ($0.2$) &   {\bf US1B} ($0.03$)\\
\hline
\hline
x= 16 &   &  &  &      \\
y & & & &  \\
\hline
1  &  {\bf LS2} ($0.0002$) & {\bf US2} ($0.03$)  & LB2 ($4$) & UP2 ($0.3$)   \\
\hline
16   &  {\bf LS2} ($0.05$)   &  {\bf US2} ($0.5$) & LB2 ($4$) & UB1 ($6$)   \\
\hline
{\bf 32}   & LB1 ($2$)    & {\bf US2} ($0.9$) & {\bf LS1} ($3$) & UP2 ($1$)   \\
\hline
{\bf 64}   & LB1 ($0.3$)    & {\bf US1A} ($0.04$) & {\bf LS1} ($0.6$) & {\bf US1B} ($0.07$)   \\
\hline
\hline
x= 32   & &  & &    \\
y &  & & & \\
\hline
1    & {\bf LS2} ($0.0002$) & {\bf US2} ($0.003$) & LB2 ($3$) & UP2 ($1$)    \\
\hline
16     & {\bf LS2} ($0.02$) & {\bf US2} ($0.3$) & LB2 ($1$) & UB1 ($10$)  \\
\hline
32    & {\bf LS2} ($0.1$) & {\bf US2} ($1$) & LB2 ($2$) & UB1 ($9$)   \\
\hline
{\bf 64}    & LB1 ($0.6$) & {\bf US1A} ($0.2$) & {\bf LS1} ($2$) & {\bf US1B} ($0.2$)  \\
\hline
\end{tabular}
\captionof{table}{Best bounds for $\mu=16$ and several values of $x$ and $y$}
\end{center}

We could consider some of our additional bounds, like for instance (\ref{betterlo}) which
is an improvement for some parameter values and, needless to say, further improvement is possible using the convergent bounds. 

\subsection{Central case}

First we compare the several bounds in Theorem \ref{resum} between them and then we compare it with other bounds
available in the literature. A numerical comparison in this case is easier to consider than in the noncentral case,
and 3D plots (not shown) are enough to get the desired information.

The second bound in Theorem \ref{resum}, $l_2$, is sharper that the first bound $l_1$. Notice that $l_1$ is given by the
first two terms in the series for $\gamma (a,y)$ in powers of $y$ \cite[8.7.1]{Paris:2010:IGR}
and that better bounds can be obtained by considering
additional terms (see also the discussion after Eq. (\ref{inhoga})); however $l_2$ is better for sufficiently large 
$y$ and $a$ even if more terms in $l_1$ are considered.

The fifth bound, $L_2$, is generally better than the fourth, $L_1$, but for small $y$ (larger as $a$ is larger) the situation is the opposite (for $A=1$). For instance, for $a=5$ $L_1$ is better when $y<5$ and for $a=100$ when $y<10$. For $A=0$ $L_1$ is
worse and only if $y<1$ (approximately) or $a$ is close to $1$ it is better than $L_2$.  

Comparing the bound $u_1 (a,y)$ with the Qi and Mei bound 
\cite[8.10.2]{Paris:2010:IGR} (but for $a>1$ with the inequality reversed), which is
\begin{equation}
\gamma (a,y)\le U_Q (a,y) = \Frac{y^{a-1}}{a}(1-e^{-y}),
\end{equation}
 we easily obtain
that $u_1$ is superior to this inequality (see also \cite{Qi:1999:SIO}) when
\begin{equation}
a>(y-1+e^{-y})/(1-(y+1)e^{-y})
\end{equation}
which approximately gives $a>y-1$ for not too small $y$. Another upper bound for $\gamma (a,y)$ can be obtained considering
$u_2 (a,y)=\Gamma (a)-L_2 (a,y)$. Numerical experiments show that that $u_2 (a,y)$ has positive values and smaller than $U_Q$ (and therefore superior) in a region containing $y>a$. Combining both bounds in a single bound yields a
bound better than $U_Q$ for any values of $a$ and $y$. Because $u_2 (a,y)$ turns out to be always positive, a simple
way to do this is by taking $u_3(a,x)=[\max \{u_1 (a,y)^{-1},u_2 (a,y)^{-1}\}]^{-1}$; $u_3(a,y)$ turns out to be superior to
$U_Q$ for any $a$ and $y$, and with relative errors with respect to $\gamma (a,y)$ always smaller that $1$ (in 
contrast to $U_Q$) and tending to zero rapidly as we move away for $y=a$.

The fourth bound for $A=0$ is the bound \cite[8.10.2]{Paris:2010:IGR} (but for $a>1$, with the inequality reversed). The case
$A=1$ is an improvement for $a>2$ and for the previous discussion it also follows that the fifth bound is generally superior
to \cite[8.10.2]{Paris:2010:IGR} (except for $y$ or $a$ small).

Finally, we compare our bounds against \cite[8.10.11]{Paris:2010:IGR} for the case $a>1$, which we write as:
\begin{equation}
\begin{array}{l}
\gamma (a,x)\ge l_H (a,x)=\Gamma (a)(1-e^{-d_a y})^a,\,d_a=(\Gamma(1+a))^{-1/a},\\
\Gamma (a,x)\ge L_H (a,x)= \Gamma (a)(1-(1-e^{-y})^a).
\end{array}
\end{equation}
We compare the lower bound $l_H$ for $\gamma (a,x)$ with the lower bounds $l_1$, $l_2$ and $l_3=\Theta(U_1)(\Gamma (a)-U_1)$
from Theorem \ref{resum} and the bound $L_H$ for $\Gamma (a,x)$ with $L_1 $, $L_2 $ and $L_3=\Theta(u1)(\Gamma (a)-u_1)$; here
$\Theta$ stands for the step function such that $\Theta(x)=1$ if $x>0$ and $\Theta (x)=0$ if $x\le 0$.

Numerical computations show that 
$l_1$ and $l_2$ are larger than $l_H$ (and then superior to $l_H$) for $y<a$, but also for $y>a$  with $y-a$ moderate. 
On the other hand, as commented before, $l_2$ is generally preferable. With respect to
$l_3$, it appears to be superior to $l_H$ for $y>a$. Combining the bounds $l_2$ and $l_3$ taking $l=\max \{l_2,l_3\}$
the resulting bound $l$ turns out to be better that $l_H$ except for a small region around $y=a$ for $y<6.5$. This region can be eliminated if the bound $l_1$ is also used, but considering additional terms in the series \cite[8.7.1]{Paris:2010:IGR}
six terms are enough).

Finally, $L_1$ and $L_2$ are larger than $L_H$ in a region containing $y>a$ while $L_3$ is larger than $L_H$ for $y<a$. As commented
before $L_2$ is generally better, and taking $L=\max\{L_2, L_3\}$ we obtain a bound which is better than $L_H$ in the whole range.

\section*{Acknowledgements}
The author acknowledges some financial support from {\emph{Ministerio de Econom\'{\i}a y Competitividad}} (project MTM2012-34787).

\end{document}